\documentclass[reqno,11pt]{amsart}
\usepackage{amsmath,amssymb,amsthm,mathrsfs}
\usepackage{epsfig,color}
\usepackage[latin1]{inputenc}
\usepackage[english]{babel}
\usepackage[numbers]{natbib}
\usepackage[colorlinks, citecolor=blue, linkcolor=red]{hyperref}
\usepackage{dsfont}

\numberwithin{equation}{section}

\newtheorem{ackn}{Acknowledgments\!}

\newcommand{\abs}[1]{{\left|#1\right|}}

\def\00{{\bf 0}}

\def\SS{\mathbb S}
\def\RR{\mathbb R}

\def\ricc{\mathrm{Ric}}

\newcommand{\eps}{{\varepsilon}}

\def\gt{\widetilde{g}}

\newtheorem*{theorem*}{Theorem}
\newtheorem{theorem}{Theorem}[section]
\newtheorem{lemma}[theorem]{Lemma}

\newtheorem{corollary}[theorem]{Corollary}
\newtheorem{remark}[theorem]{Remark}

\begin{document}
    \title[Two rigidity results for stable minimal hypersurfaces]{Two rigidity results for stable minimal hypersurfaces}

  \date{}

\author{Giovanni Catino, Paolo Mastrolia, Alberto Roncoroni}

\address{G. Catino, Dipartimento di Matematica, Politecnico di Milano, Piazza Leonardo da Vinci 32, 20133, Milano, Italy.}
\email{giovanni.catino@polimi.it}

\address{P. Mastrolia, Dipartimento di Matematica, Universit\`{a} degli Studi di Milano, Via Saldini, Italy.}
\email[]{paolo.mastrolia@unimi.it}

\address{A. Roncoroni, Dipartimento di Matematica, Politecnico di Milano, Piazza Leonardo da Vinci 32, 20133, Milano, Italy.}
\email{alberto.roncoroni@polimi.it}

\begin{abstract}
The aim of this paper is to prove two results concerning the rigidity of complete, immersed, orientable, stable minimal hypersurfaces: we show that they are hyperplane in $\mathds{R}^4$, while they do not exist in positively curved Riemannian $(n+1)$-manifold when $n\leq 5$; in particular, there are no stable minimal hypersurfaces in $\mathds{S}^{n+1}$ when $n\leq 5$. The first result was recently proved also by Chodosh and Li, and the second is a consequence of a more general result concerning minimal surfaces with finite index. Both theorems rely on a conformal method, inspired by a classical work of Fischer-Colbrie.
\end{abstract}

\maketitle

\begin{center}


\medskip

\centerline{\bf AMS subject classification:   53C42, 53C21}

\end{center}

\

\section{Introduction}

It is well-known that a minimal surface $M^2 \subset \mathds{R}^3$ is a critical point of the area functional $\mathcal{A}_t$ for all compactly supported variations, i. e. $\left.\frac{d}{d t} \right|_{t=0}\mathcal{A}_t=0$; equivalently,  $M^2$ is minimal if and only if the mean curvature $H$, i.e. the (normalized) trace of the second fundamental form,  is identically zero, or if and only if $M^2$ can be expressed, locally, as the graph $\Gamma(u)$ of a solution $u$ of the minimal surfaces equation
\[
(1+u_x^2)u_{yy}-2u_xu_yu_{xy}+(1+u_y^2)u_{xx}=0.
\]
In 1914, S. Bernstein showed that an entire (i.e., defined on the whole plane $\mathds{R}^2$) minimal graph in $\mathds{R}^3$ is necessarily a plane; the so-called ``Bernstein problem'' in higher dimension  can be then stated in the following way: if the graph $\Gamma(u)$ of a function $u: \mathds{R}^{n} \rightarrow\mathds{R}$ is a minimal hypersurface in $\mathds{R}^{n+1}$, does $\Gamma(u)$ have to be necessarily a hyperplane? Many famous mathematicians worked on this problem in the Sixties, in particular Fleming \cite{fleming} (who  gave a new proof in the case  $n=2$), De Giorgi \cite{degiorgi} (case $n=3$), Almgren \cite{almgren} (case $n=4$),  Simons \cite{simons} (the three remaining cases for  $n\leq 7$) and, eventually, Bombieri, De Giorgi and Giusti \cite{BDGG}, who showed that, for $n\geq 8$, there are minimal entire graphs that are not hyperplanes.
We explicitly remark that a minimal graph is area-minimizing, i.e. it is not only a critical point of the area functional, but also a minimum, while this is not true for minimal hypersurfaces that are ``non-graphical'', and also that area-minimizing implies \emph{stability}, that is the non-negativity of the second variation for the area functional $\left.\frac{d^2}{d t^2} \right|_{t=0}\mathcal{A}_t\geq 0$ for all compactly supported variations.

A natural generalization of the classical Bernstein problem, thus, is the \emph{stable} Bernstein problem, that is:  if $M^{n} \hookrightarrow \mathds{R}^{n+1}$ is a complete, orientable, isometrically immersed, stable minimal hypersurface, does $M$ have to be necessarily a hyperplane? In the case $n=2$ the (positive) answer was given in three different papers, which appeared between 1979 and 1981 (see do Carmo and Peng \cite{docarmopeng1979}, Fischer-Colbrie and Schoen \cite{fisSch}  and Pogorelov \cite{pogor}).

In higher dimensions, the aforementioned result of Bombieri, De Giorgi and Giusti implies that there exist non-flat orientable, complete, stable minimal hypersurfaces in $\mathds{R}^{n+1}$ for $n\geq 8$, while for $n\leq 5$ the stable Bernstein theorems is true with some additional assumptions (for instance, if one requires bounds on the volume growth of geodesic balls, see e.g. \cite{SSY}; see also \cite{docarmopeng2}, \cite{caoshenzhu}, \cite{chen}, \cite{nellisoret} and references therein for other interesting results in the same spirit). Moreover, by \cite{BDGG} and \cite{simhar}, we also note that there are non-flat area-minimizing (and thus minimal and stable) complete orientable hypersurfaces $M^{7} \hookrightarrow \mathds{R}^{8}$.

 Up until recently, without additional hypothesis, the remaining cases ($3\leq n \leq 6$) were still open, even if the study of minimal (in particular stable or in general with finite index) hypersurfaces immersed into a Riemannian manifold (not only the Euclidean space, then) is a very active field and has attracted a lot of interest.
 Then, in 2021, Chodosh and Li \cite{choli} (see also \cite{choli2}) showed that a complete, orientable, isometrically immersed,   stable minimal immersion $M^3\rightarrow \mathds{R}^4$ is a hyperplane. Their  proof, clever and highly non-trivial, is based on the non-parabolicity of $M$: they perform careful estimates for the quantity
\[
F(t) = \int_{\Sigma_t}\abs{\nabla u}^2
\]
(here $u$ is a positive Green's function  for the Laplacian and $\Sigma_t$ is the $t$-level set of $u$), relating it to $\int_{\Sigma_t}\abs{A_M}^2$ ($A_M$ is the second fundamental form of $M$).

In this paper we provide a completely different proof of Chodosh and Li result, based on a conformal deformation of the metric, a comparison result and integral estimates, and we also prove another rigidity result when the ambient space is a complete Riemannian manifold with non-negative sectional curvature and either uniformly positive bi-Ricci curvature or uniformly positive Ricci curvature.
To be precise, and to fix the notation, we consider smooth, complete, connected, orientable, isometrically immersed  hypersurfaces $M^n\hookrightarrow (X^{n+1},h)$, $n\geq 2$, where $(X^{n+1},h)$ is a (complete) Riemannian manifold of dimension $n+1$ endowed with metric $h$. We denote with $g$ the induced metric on $M$ and with $H$ the mean curvature of $M$; we have that $M$ is {\em minimal} if $H\equiv 0$ on $M$. In this latter case we say that $M$ is {\em stable} if
\begin{equation}\label{eq-st}
\int_M \left[|A|^2+\ricc_h(\nu,\nu)\right] \varphi^2\,dV_g \leq \int_M |\nabla \varphi|^2\,dV_g \qquad\forall \varphi\in C^{\infty}_0(M),
\end{equation}
where $A=A_M$ is the second fundamental form of $M^n$, $\nu$ is a unit normal vector to $M$ in $X$ and $dV_g$ is the volume form of $g$ . 

As we recalled before, stability is related to the non-negativity of the second variation or, equivalently, the non-positivity of the Jacobi operator
$$
L_M:=\Delta +|A|^2+\ricc_h(\nu,\nu).
$$

%
%

The first result is thus the following:
\begin{theorem}\label{t1} A complete, orientable, immersed, {\em stable} minimal hypersurface $M^3\hookrightarrow \mathds{R}^{4}$ is a hyperplane.
\end{theorem}

\

The second result concerns minimal hypersurfaces with {\em finite index}. We recall that a minimal immersion $M^n\hookrightarrow (X^{n+1},h)$ has finite index if the number of negative eigenvalues (counted with multiplicity)  of the Jacobi operator $L_M$ on every compact domain in $M$ with Dirichlet boundary conditions  is finite; in particular stability implies finite (equal zero) index.  Before presenting our next result, we need to recall the notion of \emph{bi-Ricci curvature tensor} introduced in \cite{ShYe}: given two orthonormal tangent vectors $u,v$ we define
$$
\mathrm{BRic}_h(u,v)=\mathrm{Ric}_h(u,u)+\mathrm{Ric}_h(v,v) - \mathrm{Sect}_h(u,v)\, ,
$$
where $\mathrm{Sect}_h(u,v)$ denotes the sectional curvature of the plane spanned by $u$ and $v$.
Our second result is the following:

\begin{theorem}\label{t2} If $(X^{n+1},h)$ is a complete $(n+1)$-dimensional, $n\leq 5$, manifold with non-negative sectional curvature and either uniformly positive bi-Ricci curvature or uniformly positive Ricci curvature, then  every complete, orientable, immersed, minimal hypersurface $M^n\hookrightarrow (X^{n+1},h)$ with {\em finite index} must be compact.
\end{theorem}

As a byproduct we have the 

\begin{corollary}\label{cor} If $(X^{n+1},h)$ is a complete $(n+1)$-dimensional, $n\leq 5$, manifold with non-negative sectional curvature and uniformly positive Ricci curvature, then there is no complete, orientable, immersed, {\em stable} minimal hypersurface $M^n\hookrightarrow(X^{n+1},h)$.
\end{corollary}

In particular, there is no complete, orientable, immersed, stable minimal hypersurface of the round spheres $M^n\hookrightarrow (\SS^{n+1},g_{\text{std}})$, provided $n\leq 5$. In dimension $n=2$ this follows from a more general result proved in \cite{schyau}, while, in dimension $n=3$, it was recently proved in \cite[Corollary 1.5]{cholistr}.  We mention that Theorem \ref{t2} holds also for complete, orientable, immersed, stable minimal hypersurface of the cylinder $M^n\hookrightarrow (\RR\times\SS^{n},g_{\text{std}})$ (observe that in this case $\mathrm{Sect}\geq 0 $ and $\mathrm{BRic}\geq 1$), provided $n\leq 5$. As far as we know, Corollary \ref{cor} is new in the cases $n=4,5$. We do not know if Theorem \ref{t2} and Corollary \ref{cor} hold also in dimension greater than five. We note that, in the same spirit, in \cite{ShYe} the authors obtained a compactness result for stable minimal hypersurfaces of dimension $n\leq 4$ immersed in space with uniformly positive bi-Ricci curvature.

\

\section{Proof of Theorem \ref{t1}}

In this section we give an alternative proof of \cite[Theorem 1]{choli} (see Theorem \ref{t1}). The main idea is to use a weighted volume comparison for a suitable conformal metric $\gt$ together with a new weighted integral estimate inspired by \cite{SSY}. 

Let $M^n\hookrightarrow \mathds{R}^{n+1}$ be a complete, connected, orientable, isometrically immersed, stable minimal hypersurface.

\subsection{Conformal change} It is well known (see e.g. \cite[Proposition 1]{fis}) that, since $M^n\hookrightarrow \mathds{R}^{n+1}$ is stable, then there exists a positive function $0<u\in C^{\infty}(M)$ satisfying
\begin{equation}\label{eq-u}
-\Delta_g u = |A|_g^2 u\quad\text{on }M\,.
\end{equation}
Following the line in \cite{fis} (see also \cite{elbnelros}), let $k>0$ and consider the conformal metric
$$
\gt = u^{2k} g,
$$
where $g=\imath^*h$ is the induced metric on $M$ (and $\imath$ denotes the inclusion). First of all we prove the following lower bound for a modified Bakry-Emery-Ricci curvature of $\gt$. In particular, this implies the non-negativity of the $2$-Bakry-Emery-Ricci curvature of $\gt$ for a suitable $k$.
\begin{lemma}\label{l-con} Let $f:=k(n-2)\log u$. Then the Ricci tensor of the metric $\gt=u^{2k} g$ satisfies
$$
\mathrm{Ric}_{\gt}+\nabla^2_{\gt}f-\frac{1-k(n-2)}{k(n-2)^2}df\otimes df \geq \left(k-\frac{n-1}{n}\right)|A|_g^2 g
$$
in the sense of quadratic forms.
In particular, if $n=3$ and $k=\frac 23$, then the $2$-Bakry-Emery-Ricci tensor $\mathrm{Ric}_{\gt}^{2,f}:=\mathrm{Ric}_{\gt}+\nabla^2_{\gt}f-\frac12 df\otimes df$ satisfies
$$
\mathrm{Ric}_{\gt}^{2,f} \geq 0.
$$
\end{lemma}
\begin{proof}
  Since $f=k(n-2)\log u$, we  have
  \[
  df = k(n-2)\frac{du}{u}
   \]
   and
   \[
   \nabla^2_gf= k(n-2)\left(\frac{\nabla^2_gu}{u}-\frac{du\otimes du}{u^2}\right),
   \]
   which implies
   \[
   \Delta_gf = k(n-2)\left(\frac{\Delta_gu}{u}-\frac{\abs{\nabla_gu}^2_g}{u^2}\right).
   \]
    On the other hand, from the standard formulas for a conformal change of the metric $\tilde{g}=e^{2\varphi}g$, $\varphi\in C^\infty(M)$, $\varphi>0$ we get
 \[
  \mathrm{Ric}_{\gt} = \mathrm{Ric}_g-(n-2)\left(\nabla_g^2\varphi-d\varphi\otimes d\varphi\right)-\left[\Delta_g \varphi+(n-2)\abs{\nabla_g \varphi}_g^2\right]g
  \]
  and
    \[
  \nabla^2_{\gt} f = \nabla^2_gf-\left(df\otimes d\varphi+d\varphi\otimes df\right)+g\left(\nabla f, \nabla \varphi\right)g.
  \]
  Note that, in our case, $\varphi=k\log u$; now we exploit the facts that $u$ is a solution of equation \eqref{eq-u} to write
  \begin{align*}
  \mathrm{Ric}_{\gt}+\nabla^2_{\gt} f &=  \mathrm{Ric}_g-k^2(n-2)\frac{du\otimes du}{u^2} + k\abs{A}_g^2g + k\frac{\abs{\nabla_gu}^2_g}{u^2}g \\ &=\mathrm{Ric}_g-\frac{df\otimes df}{n-2} + k\abs{A}^2g + \frac{\abs{\nabla_gf}^2_g}{k(n-2)^2}g.
  \end{align*}
 From the Cauchy-Schwarz inequality we have
 \[
 \abs{\nabla_gf}^2_gg \geq df\otimes df,
 \]
 thus
   \begin{align*}
  \mathrm{Ric}_{\gt}+\nabla^2_{\gt} f &\geq \mathrm{Ric}_g+\frac{1-k(n-2)}{k(n-2)^2}df\otimes df + k\abs{A}_g^2g;
  \end{align*}
 from Gauss equations in the minimal case we get $\mathrm{Ric}_g=-A^2$; since $A$ is traceless we have the inequality
 \[
 A^2 \leq \frac{n-1}{n}\abs{A}^2g,
 \]
 and substituting in the previous relation we conclude
\[
\mathrm{Ric}_{\gt} + \nabla^2_{\gt} f-\frac{1-k(n-2)}{k(n-2)^2}df\otimes df \geq \left(k-\frac{n-1}{n}\right)|A|_g^2 g.
\]
\end{proof}

\

\subsection{Completeness} In this subsection we are going to prove that the conformal metric $\gt = u^{2k} g$ is complete, provided $n=3$ and $k=\frac 23$. In order to do this we follow the strategy in \cite{fis} and we use some computations in \cite{elbnelros}. First, we recall that in the proof of \cite[Theorem 1]{fis}, given a reference point $O\in M^n$, the author showed the existence of a $\gt$-minimizing geodesic, 
$$
\gamma(s):[0,\infty)\to M^n,
$$
where $s$ is the $g$-arclength and $M^n\hookrightarrow \mathds{R}^{n+1}$ is the usual complete, connected, orientable, isometrically immersed, stable minimal hypersurface. For the sake of completeness, we report the argument here. First of all, for every $R>0$, we consider the geodesic ball (of $g$) centered at $O$ of radius $R$, $B_R(O)$. Then, we first claim that there exists a $\gt$-minimizing geodesic, $\gamma_R$, joining $O$ to any boundary point of $B_R(O)$. Indeed, consider $u_R:=u+\eta$, where $\eta$ is a smooth function such that $\eta\equiv 0$ in $B_R(O)$ and $\eta\equiv 1$ in $M^n\setminus B_{R+1}(O)$. Since $u_R$ is bounded below, the metric 
$$
\gt_R=u_R^{\frac{2(n-1)}{n}}g
$$ 
is complete, and these geodesics exist. Therefore, for every $R_i>0$, since $\partial B_{R_i}(O)$ is compact, there exists $x_i\in \partial B_{R_i}(O)$ so that $x_i$ is closest (in $\gt_R$) to $O$. Let $\gamma_i$ be the $\gt_R$-minimizing geodesic joining $O$ to $x_i$. Note that $\gamma_i\subset B_{R_i}(O)$ or another point would be closer to $O$. Since $u_{R_i}=u$ in $B_{R_i}(O)$, then $\gamma_i$ is a $\gt$-minimizing geodesic.
We parametrize $\gamma_i$ with respect to $g$-arclength. In particular, since $|\dot \gamma_i(s)|_g=1$ for every $s$, up to subsequences, the sequence $ \dot\gamma_i(0)$ converges to a limit vector as $R_i\to \infty$. Thus, by ODE theory and Ascoli-Arzel\`a, $\gamma_i$ converges on compact sets of $[0,\infty)$ to a limiting curve $\gamma$ which is a $\gt$-minimizing geodesic ad is parametrized by $g$-arclength.

\begin{remark}\label{r-lunghe}
$(i)$ We observe that the completeness of the metric $\gt = u^{2k} g$ will follow if we can show that the $\gt$-length of $\gamma$ is infinite, i.e.
$$
\int_\gamma\,d\tilde{s} = \int_\gamma u^{k}\,ds=+\infty.
$$
Indeed, by construction, the $\gt$-length of every other divergent geodesic starting from $O$ (i.e. its image does not lie in any ball $B_R(O)$) must be greater or equal than the one of $\gamma$.

\medskip

\noindent $(ii)$ Note that $\gamma$ has unit speed with respect to $g$ and to $\gt$, when it is parametrized by the arclength $s$ and $\tilde{s}$, respectively.
\end{remark}

\

\noindent From now on 
$$
n=3 \quad \text{and} \quad k=\frac{2}{3}\, . 
$$



\begin{lemma}\label{l-com} The metric $\gt=u^{\frac43} g$ is complete.
\end{lemma}
\begin{proof} We do part of the computations for every $n$. We consider the $\gt$-minimizing geodesic $\gamma$ just constructed and as observed in Remark \ref{r-lunghe} the completeness of $\gt$ is equivalent to prove that $\gamma$ has infinite $\gt$ length, i.e.
$$
\int_0^{+\infty}u^{k}(\gamma(s))\,ds =+\infty.
$$
Since $\gamma$ is minimizing, by the second variation formula, following the computations in the proof of Theorem 1 (with $H=0$) in \cite{elbnelros}, we obtain
\begin{align*}
(n-1)\int_0^{+\infty}& (\varphi_s)^2u^{-k}\,ds \geq \int_0^{+\infty} \varphi^2 u^{-k}\left(k|A|^2-A^2_{11}-\sum_{j=2}^n A^2_{1j}\right)\,ds\\
&\quad-k(n-2)\int_0^{+\infty} \varphi^2u^{-k}(\log u)_{ss}\,ds +k\int_0^{+\infty}\varphi^2u^{-k}\frac{|\nabla u|^2}{u^2}\,ds
\end{align*}
for every smooth function $\varphi$ with compact support in $(0,+\infty)$ and for every $k>0$. Since $A$ is trace-free, we have
$$
|A|^2\geq A_{11}^2+A_{22}^2+\ldots+A_{nn}^2+2\sum_{j=2}^n A_{1j}^2 \geq \frac{n}{n-1}A_{11}^2+2\sum_{j=2}^n A_{1j}^2,
$$
thus
$$
k|A|^2-A^2_{11}-\sum_{j=2}^n A^2_{1j}\geq \left(\frac{kn}{n-1}-1\right)A_{11}^2+(2k-1)\sum_{j=2}^n A^2_{1j}.
$$
In particular, if
\begin{equation}\label{k1}
k\geq \frac{n-1}{n}
\end{equation}
we have
$$
\int_0^{+\infty} \varphi^2 u^{-k}\left(k|A|^2-A^2_{11}-\sum_{j=2}^n A^2_{1j}\right)\,ds\geq 0.
$$
Using this estimate, the fact that $|\nabla u|^2\geq (u_s)^2$ and integrating by parts, we obtain
\begin{align*}
(n-1)\int_0^{+\infty} (\varphi_s)^2u^{-k}\,ds &\geq
2k(n-2)\int_0^{+\infty} \varphi\varphi_s u^{-k-1}u_s\,ds \\
&\quad+k\left[1-k(n-2)\right]\int_0^{+\infty}\varphi^2u^{-k-2}(u_s)^2\,ds.
\end{align*}
Let now  $\varphi=u^k\psi$, with $\psi$ smooth with compact support in $(0,+\infty)$. We have
\begin{align*}
  \varphi^2u^{-k}&=u^{k}\psi^2,\\ \varphi_s&=k\psi u^{k-1}u_s+u^k\psi_s,\\ (\varphi_s)^2u^{-k}&=k^2\psi^2u^{k-2}(u_s)^2+u^k(\psi_s)^2+2k\psi\psi_su^{k-1}u_s,
\end{align*}
and substituting in the previous relation we get
\begin{align}\label{es1}
(n-1)\int_0^{+\infty} u^{k}(\psi_s)^2\,ds \geq
-2k\int_0^{+\infty} \psi\psi_s u^{k-1}u_s\,ds +k(1-k)\int_0^{+\infty}\psi^2u^{k-2}(u_s)^2\,ds.
\end{align}
Let
$$
I:=\int_0^{+\infty} \psi\psi_s u^{k-1}u_s\,ds;
$$
thus we have
$$
I = \frac 1k \int_0^{+\infty} \psi\psi_s (u^{k})_s\,ds=-\frac 1k \int_0^{+\infty} u^k(\psi_s)^2\,ds-\frac 1k \int_0^{+\infty} \psi\psi_{ss} u^{k}\,ds.
$$
Moreover, for every $t>1$ and using Young's inequality for every $\eps>0$, we have
\begin{align*}
2kI &=2kt I+2k(1-t)I\\ &=-2t \int_0^{+\infty} u^k(\psi_s)^2\,ds-2t \int_0^{+\infty} \psi\psi_{ss} u^{k}\,ds+2k(1-t)\int_0^{+\infty} \psi\psi_s u^{k-1}u_s\,ds\\
&\leq -2t \int_0^{+\infty} u^k(\psi_s)^2\,ds-2t \int_0^{+\infty} \psi\psi_{ss} u^{k}\,ds\\
&\quad+k(t-1)\eps \int_0^{+\infty}\psi^2u^{k-2}(u_s)^2\,ds+\frac{k(t-1)}{\eps}\int_0^{+\infty} u^k(\psi_s)^2\,ds.
\end{align*}

Assuming
\begin{equation}\label{k2}
k<1
\end{equation}
and choosing
$$
\eps:=\frac{1-k}{t-1}
$$
we obtain
\begin{align*}
2kI &\leq -2t \int_0^{+\infty} \psi\psi_{ss} u^{k}\,ds+k(1-k)\int_0^{+\infty}\psi^2u^{k-2}(u_s)^2\,ds\\
&\quad+\left[\frac{k(t-1)^2}{1-k}-2t\right]\int_0^{+\infty} u^k(\psi_s)^2\,ds.
\end{align*}
From \eqref{es1} we get
$$
0\leq \left[\frac{k(t-1)^2}{1-k}-2t+(n-1)\right]\int_0^{+\infty} u^k(\psi_s)^2\,ds-2t \int_0^{+\infty} \psi\psi_{ss} u^{k}\,ds
$$
for every $t>1$ and every $k$ satisfying \eqref{k1} and \eqref{k2}. Let
$$
P(t):=\frac{k(t-1)^2}{1-k}-2t+(n-1)
$$
and choose $k=\frac{n-1}{n}$. It is easy to see that $P(t)$ is negative for some $t>1$ if $n=3$: indeed
$$
P(t)=(n-1)t^2-2nt+2(n-1)=2t^2-6t+4=-2(t-1)(2-t).
$$
Therefore, if $n=3$, $k=\frac 23$ and $t=\frac 32$, we deduce
$$
0\leq -\int_0^{+\infty} u^{\frac23}(\psi_s)^2\,ds-6 \int_0^{+\infty} u^{\frac23}\psi\psi_{ss} \,ds
$$
for every $\psi$ smooth with compact support in $(0,+\infty)$. Now we choose $\psi=s\eta$ with $\eta$ smooth with compact support in $(0,+\infty)$: thus
$$
\psi_s=\eta+s\eta_s,\quad\psi_{ss}=2\eta_s+s\eta_{ss},
$$
and we get
$$
\int_0^{+\infty} u^{\frac23}\eta^2\,ds\leq \int_0^{+\infty} u^{\frac23}\left(-14s\eta\eta_s-6s^2\eta\eta_{ss}-s^2(\eta_s)^2\right)\,ds.
$$
Choose $\eta$ such that $\eta\equiv 1$ on $[0,R]$, $\eta\equiv 0$ on $[2R,+\infty)$ and with $|\eta_s|$ and $|\eta_{ss}|$ bounded by $C/R$ and $C/R^2$, respectively, for $R\leq s\leq 2R$ ($C$ is a positive constant). Then
$$
\int_0^R u^{\frac23}\,ds\leq \int_0^{+\infty} u^{\frac23}\eta^2\,ds \leq C \int_R^{+\infty}u^{\frac23}\,ds
$$
for some $C>0$ independent of $R$. We conclude that
$$
\int_0^{+\infty}u^{\frac23}\,ds =+\infty,
$$
i.e. $\gt=u^{\frac 43} g$ is complete.
\end{proof}

\subsection{Weighted integral estimates}  From lemma \ref{l-con} and lemma \ref{l-com} we have that the metric $\gt=u^{\tfrac43}g$ is complete and it has non-negative $2$-Bakry-Emery-Ricci curvature. Using well known comparison results (see \cite{Qian}) we immediately obtain the following weighted Bishop-Gromov volume estimate for a geodesic ball $B^{\gt}_R(x_0)$ centered at $x_0\in M$, of radius $R$, with respect to the metric $\gt$.

\begin{corollary}\label{c-bg} Let $x_0\in M^3$. Then, for every $R>0$, there exists $C>0$ such that the $f$-volume
$$
\operatorname{Vol}_f B^{\gt}_R(x_0):=\int_{B^{\gt}_R(x_0)}e^{-f}\,dV_{\gt}\leq C R^5,
$$
where $f=\frac23 \log u$. Equivalently, in terms of $u$ and the volume form of $g$,
$$
\int_{B^{\gt}_R(x_0)}u^{\frac 43}\,dV_{g}\leq C R^5.
$$
\end{corollary}

The last ingredient that we need in the proof of Theorem \ref{t1} is the following weighted integral inequality
in the spirit of \cite[Theorem 1]{SSY}.

\begin{lemma}\label{l-ssy} For every $0<\delta<\frac{1}{100}$, there exists $C>0$ such that
$$
\int_M |A|^{5+\delta} u^{-2-\frac{2\delta}{3}} \psi^{5+\delta}\,dV_g \leq C \int_M  u^{-2-\frac{2\delta}{3}} |\nabla \psi|^{5+\delta}\,dV_g\qquad\forall \psi\in C^{\infty}_0(M).
$$
\end{lemma}

\begin{proof}
Again, we do part of the computations for every $n$. From \cite{SSY} we get
\begin{equation}\label{SSY_1}
\int_M |A|^p  \varphi^2 \leq C \int_M  |A|^{p-2}|\nabla \varphi|^2\qquad\forall \varphi\in C^{\infty}_0(M),
\end{equation}
for every $p\in[4,4+\sqrt{8/n}]$ and for some $C=C(n,p)>0$.  For the sake of completeness we report here the proof of \eqref{SSY_1}.  We take $\varphi=|A|^{1+q}\psi$, $q\geq 0$, with $\psi\in C^{\infty}_0(M)$, in the stability inequality \eqref{eq-st} obtaining
$$
\int_M |A|^{4+2q}\psi^2 \leq [(1+q)^2+\eps]\int_M |A|^{2q}|\nabla|A||^2\psi^2 + \frac{1+q}{\eps}\int_M |A|^{2+2q}|\nabla \psi|^2,
$$
for every $\eps>0$, where we used Young's inequality. On the other hand, multiplying Simons' inequality (see \cite[Lemma 2.1]{colmin} for a proof)
\begin{equation}\label{simons} 
|A|\Delta|A|+|A|^4\geq \frac{2}{n}|\nabla|A||^2.
\end{equation}
by $|A|^{2q}\psi^2$ and integrating by parts, we get
$$
\left(\frac 2n +1+2q-\eps\right)\int_M |A|^{2q}|\nabla|A||^2\psi^2 \leq  \int_M |A|^{4+2q}\psi^2 + \frac{1}{\eps}\int_M |A|^{2+2q}|\nabla \psi|^2
$$
for every $\eps>0$, where we used again Young's inequality. Since $q\geq 0$,  for $\eps>0$ sufficiently small, we obtain
$$
\left\{1- [(1+q)^2+\eps]\left(\frac 2n +1+2q-\eps\right)^{-1}\right\}\int_M |A|^{4+2q}\psi^2 \leq C\int_M |A|^{2+2q}|\nabla \psi|^2.
$$
Let $q:=\frac{p-4}{2}$. For $\eps>0$ small enough, we have
$$
1- [(1+q)^2+\eps]\left(\frac 2n +1+2q-\eps\right)^{-1}>0
$$
if $p\in[4,4+\sqrt{8/n}]$ and we finally obtain
$$
\int_M |A|^p  \psi^2 \leq C \int_M  |A|^{p-2}|\nabla \psi|^2\qquad\forall \psi\in C^{\infty}_0(M).
$$
Taking $\psi=\varphi^{p/2}$, by Holder's inequality we get \eqref{SSY_1}.

Take $\varphi=u^{\alpha}\psi$, with $\psi$ smooth with compact support, $u$ the solution of \eqref{eq-u} and $\alpha<0$.  Since, from Cauchy-Schwarz and Young's inequalities,
\begin{equation}\label{grad_product}
\vert\nabla (u^{\alpha}\psi)\vert^2 \leq 2 \psi^2 \vert\nabla (u^{\alpha})\vert^2 + 2 u^{2\alpha}\vert\nabla \psi\vert^2 \,,
\end{equation}
then \eqref{SSY_1} becomes
\begin{equation}\label{SSY_2}
\int_M |A|^p  u^{2\alpha}\psi^2 \leq 2C\left[ \int_M  |A|^{p-2}\psi^2|\nabla u^\alpha|^2 + \int_M  |A|^{p-2}u^{2\alpha}|\nabla \psi|^2 \right]   \quad\forall \psi\in C^{\infty}_0(M).
\end{equation}
Now we tackle the first integral on the right-hand side of \eqref{SSY_2} firstly integrating by parts
\begin{align*}
\int_M  |A|^{p-2}\psi^2|\nabla u^\alpha|^2= -\int_M  |A|^{p-2}\psi^2 u^\alpha\Delta u^{\alpha}- \int_M  u^\alpha\psi^2 \langle \nabla u^\alpha, \nabla  |A|^{p-2}\rangle\\
- 2\int_M |A|^{p-2} u^\alpha\psi  \langle \nabla u^\alpha, \nabla  \psi\rangle\, ,
\end{align*}
secondly we use the fact that
$$
\Delta u^\alpha=\alpha u^{\alpha-1}\Delta u + \alpha(\alpha-1)u^{\alpha-2}\vert\nabla u\vert^2\quad \text{and} \quad \vert\nabla u^\alpha\vert^2=\alpha^2 u^{2\alpha-2}\vert\nabla u\vert^{2} \, ,
$$
together with Cauchy-Schwarz and Young's inequalities to get
\begin{align*}
\int_M  |A|^{p-2}\psi^2|\nabla u^\alpha|^2\leq - \alpha\int_M |A|^{p-2}u^{2\alpha-1}\psi^2\Delta u - \dfrac{\alpha-1}{\alpha}\int_M |A|^{p-2} \psi^2 \vert\nabla u^\alpha\vert^2 \\
 - \int_M  u^\alpha\psi^2 \langle \nabla u^\alpha, \nabla  |A|^{p-2}\rangle + \eps\int_M |A|^{p-2}\psi^2\vert\nabla u^\alpha\vert^2  + \frac{1}{\eps}\int_M |A|^{p-2} u^{2\alpha} \vert\nabla \psi\vert^2\, ,
\end{align*}
for all $\varepsilon>0$. From \eqref{eq-u} we find
\begin{align*}
\int_M  |A|^{p-2}\psi^2|\nabla u^\alpha|^2\leq \alpha\int_M |A|^{p}u^{2\alpha}\psi^2- \dfrac{\alpha-1}{\alpha}\int_M |A|^{p-2} \psi^2 \vert\nabla u^{\alpha}\vert^2 \\
 - \int_M  u^\alpha\psi^2 \langle \nabla u^\alpha, \nabla  |A|^{p-2}\rangle + \eps\int_M |A|^{p-2}\psi^2\vert\nabla u^\alpha\vert^2  + \frac{1}{\eps}\int_M |A|^{p-2} u^{2\alpha} \vert\nabla \psi\vert^2\, ,
\end{align*}
i.e.
\begin{align*}
\left( 1-\eps+ \frac{\alpha-1}{\alpha} \right)\int_M  |A|^{p-2}\psi^2|\nabla u^\alpha|^2\leq  \alpha\int_M |A|^{p}u^{2\alpha}\psi^2\\- \int_M  u^\alpha\psi^2 \langle \nabla u^\alpha, \nabla  |A|^{p-2}\rangle + \frac{1}{\eps}\int_M |A|^{p-2} u^{2\alpha} \vert\nabla \psi\vert^2\, ,
\end{align*}
Now, since
$$
\nabla  |A|^{p-2}=(p-2) \vert A\vert^{p-3}\nabla  |A|=(p-2) \vert A\vert^{\frac{p-2}{2}}\vert A\vert^{\frac{p-4}{2}}\nabla  |A|\, ,
$$
then, from Cauchy-Schwarz and Young's inequalities we obtain
\begin{align}\label{eq-s2}
\left( 1-\eps+ \frac{\alpha-1}{\alpha}-\frac{p-2}{2t_1} \right)\int_M  |A|^{p-2}\psi^2|\nabla u^\alpha|^2\leq \alpha\int_M |A|^{p}u^{2\alpha}\psi^2 \\ \nonumber+\frac{(p-2)t_1}{2}\int_M |A|^{p-4}\psi^2 u^{2\alpha} \vert\nabla \vert A\vert\vert^2 + \frac{1}{\eps}\int_M |A|^{p-2} u^{2\alpha} \vert\nabla \psi\vert^2\, ,
\end{align}
for every $t_1>0$. Now, multiplying by $|A|^{p-4}f^2$ the Simons' inequality \eqref{simons}, integrating by parts and using Young's inequality we obtain
\begin{align*}
\int_M |A|^p f^2 \geq \left(\frac2n+p-3-t_2\right)\int_M |A|^{p-4}|\nabla|A||^2f^2-\frac{1}{t_2}\int_M|A|^{p-2}|\nabla f|^2
\end{align*}
for every $t_2>0$. Choosing $f=u^{\alpha}\psi$ we get
\begin{align}\label{eq-gianny}
\int_M |A|^p u^{2\alpha}\psi^2 &\geq \left(\frac2n+p-3-t_2\right)\int_M |A|^{p-4}|\nabla|A||^2 u^{2\alpha}\psi^2\\\nonumber
&\quad-\left(\frac{1}{t_2}+\eps\right)\int_M  |A|^{p-2}\psi^2|\nabla u^\alpha|^2 -\frac{1}{t_2}\left(1+\frac{1}{t_2\eps}\right)\int_M  |A|^{p-2}u^{2\alpha}|\nabla \psi|^2
\end{align}
for every $\eps>0$, since
$$
\vert\nabla (u^{\alpha}\psi)\vert^2 \leq (1+t_2\eps) \psi^2 \vert\nabla (u^{\alpha})\vert^2 + \left(1+\frac{1}{t_2\eps}\right) u^{2\alpha}\vert\nabla \psi\vert^2 \,.
$$
 Now let $\delta>0$. Using \eqref{eq-gianny} in \eqref{eq-s2} with $$\alpha=-1-\frac{\delta}{3}\leq-1$$ we obtain
\begin{align*}
&\left(1+\frac{2+\frac{\delta}{3}}{1+\frac{\delta}{3}}-\left(2+\frac{\delta}{3}\right)\eps-\frac{p-2}{2t_1}-\frac{1+\frac{\delta}{3}}{t_2} \right)\int_M  |A|^{p-2}\psi^2|\nabla u^{-1-\frac{\delta}{3}}|^2\\
&\,\leq  \left[\frac{1}{\eps}+\frac{1+\frac{\delta}{3}}{t_2}\left(1+\frac{1}{t_2\eps}\right)\right]\int_M |A|^{p-2} u^{-2-\frac{2\delta}{3}} \vert\nabla \psi\vert^2 \\ 
&\,+\left[\frac{(p-2)t_1}{2}-\frac{2+\frac{2\delta}{3}}{n}-\left(1+\frac{\delta}{3}\right)p+3+\delta+\left(1+\frac{\delta}{3}\right)t_2\right]\int_M |A|^{p-4}\psi^2 u^{-2-\frac{2\delta}{3}} \vert\nabla \vert A\vert\vert^2 \,,
\end{align*}
for all $\eps,t_1,t_2>0$. Let
$$
n=3, \quad p=5+\delta,\quad t_1=\frac{2(5+3\delta)}{9},\quad t_2=1
$$
we obtain
$$
\frac{(p-2)t_1}{2}-\frac{2+\frac{2\delta}{3}}{n}-\left(1+\frac{\delta}{3}\right)p+3+\delta+\left(1+\frac{\delta}{3}\right)t_2=0
$$
and
$$
1+\frac{2+\frac{\delta}{3}}{1+\frac{\delta}{3}}-\frac{p-2}{2t_1}-\frac{1+\frac{\delta}{3}}{t_2}=\frac{603+378\delta+7\delta^2-12\delta^3}{180+168\delta+36\delta^2}.
$$
Thus
\begin{align*}
&\left( \frac{603+378\delta+7\delta^2-12\delta^3}{180+168\delta+36\delta^2}-\left(2+\frac{\delta}{3}\right)\eps \right)\int_M  |A|^{3+\delta}\psi^2|\nabla u^{-1-\frac{\delta}{3}}|^2\\
&\qquad\leq  \left[\frac{1}{\eps}+\left(1+\frac{\delta}{3}\right)\left(1+\frac{1}{\eps}\right)\right]\int_M |A|^{3+\delta} u^{-2-\frac{2\delta}{3}} \vert\nabla \psi\vert^2 ,
\end{align*}
for all $\eps>0$. Choosing  $0<\delta<1/100$ and $\eps$ small enough we obtain
$$
\int_M  |A|^{3+\delta}\psi^2|\nabla u^{-1-\frac{\delta}{3}}|^2\leq  C\int_M |A|^{3+\delta} u^{-2-\frac{2\delta}{3}} \vert\nabla \psi\vert^2\,,
$$
for some $C>0$. From \eqref{SSY_2} and Young's inequality we get
\begin{align*}
\int_M |A|^{5+\delta}  u^{-2-\frac{2\delta}{3}}\psi^2 &\leq C\int_M  |A|^{3+\delta}u^{-2-\frac{2\delta}{3}}|\nabla \psi|^2 \\
&\leq \eps' \int_M |A|^{5+\delta}  u^{-2-\frac{2\delta}{3}}\psi^2+\frac{C}{\eps'}\int_M u^{-2-\frac{2\delta}{3}}|\nabla \psi|^{5+\delta}\psi^{-(3+\delta)}
\end{align*}
for all $\eps'>0$ and $\psi\in C^{\infty}_0(M)$. Therefore
\begin{align*}
\int_M |A|^{5+\delta}  u^{-2-\frac{2\delta}{3}}\psi^2 &\leq C\int_M u^{-2-\frac{2\delta}{3}}|\nabla \psi|^{5+\delta}\psi^{-(3+\delta)}\\
&=C\int_M u^{-2-\frac{2\delta}{3}}|\nabla \psi^{\frac{2}{5+\delta}}|^{5+\delta}.
\end{align*}
The conclusion now follows immediately by replacing $\psi$ with $\psi^{\frac{2}{5+\delta}}$.
\end{proof}

\subsection{Final estimate} Combining Lemma \ref{l-ssy} with Corollary \ref{c-bg} we can conclude the proof of Theorem \ref{t1}. More precisely, let $x_0\in M$ and let $\widetilde{r}$ the distance function from $x_0$ with respect to the metric $\gt=u^{\frac 43} g$. We choose  $\psi:=\eta(\widetilde{r})$ with $0\leq\eta\leq 1$, $\eta\equiv 1$ on $[0,R]$, $\eta\equiv 0$ on $[2R,+\infty)$ and $|\eta'|\leq C/R$ on $[R,2R]$, for some $C>0$ and $R>0$. From Lemma \ref{l-ssy}, for some $0<\delta<1/100$, we have
\begin{align*}
\int_M |A|^{5+\delta} u^{-2-\frac{2\delta}{3}} \eta^{5+\delta}\,dV_g &\leq C \int_M  u^{-2-\frac{2\delta}{3}} |\nabla \psi|_g^{5+\delta} \,dV_g\\
&= C \int_M  u^{-2-\frac{2\delta}{3}+\frac{2(5+\delta)}{3}} |\tilde\nabla \psi|_{\gt}^{5+\delta}\,dV_g\\
&\leq\frac{C}{R^{5+\delta}}\int_{B^{\gt}_{2R}(x_0)}  u^{\frac{4}{3}}\,dV_g\\
&\leq \frac{C}{R^\delta},
\end{align*}
where we used the fact that $|\tilde\nabla \widetilde{r}|_{\gt}\equiv 1$ and Corollary \ref{c-bg}. Since $\delta>0$, letting $R\to+\infty$ we get
$$|A|\equiv 0\quad\text{on } M^3$$ and this concludes the proof of  Theorem \ref{t1}.

\

\section{Proof of Theorem \ref{t2}}

\begin{proof}[Proof of Theorem \ref{t2}] Let $(X^{n+1},h)$ be a complete $n$-dimensional, $n\leq 5$, manifold with non-negative sectional curvature and either uniformly positive bi-Ricci curvature or uniformly positive Ricci curvature and consider an orientable, immersed, minimal hypersurface $M^n\to (X^{n+1},h)$ with finite index. Suppose, by contradiction, that $M$ is non-compact. It is well known (see \cite[Proposition 1]{fis}) that there exist $0<u\in C^{\infty}(M)$ and a compact subset $K\subset M$ such that $u$ solves
$$
-\Delta u = \left[|A|^2+\ricc_h(\nu,\nu)\right] u\qquad\text{on } M\setminus K.
$$
Let $k>0$ and consider the conformal metric
$$
\gt = u^{2k} g.
$$
where $g$ is the induced metric on $M$. Let  $s$ be the arc length with respect to the metric $g$. Following the construction in \cite[Theorem 1]{fis}, we can construct a minimizing geodesic $\widetilde\gamma(s):[0,+\infty)\to M\setminus K$ in the metric $\gt$ which has infinite length in the metric $g$. Now we can argue exactly as in the proof of estimate $(6)$ in \cite{elbnelros}, using $H\equiv 0$, obtaining
\begin{align*}
(n-1)\int_0^a (\varphi_s)^2\,ds &\geq k(n-3)\int_0^a \varphi\varphi_s\frac{u_s}{u}\,ds +\frac{k\left[4-k(n-1)\right]}{4}\int_0^a\varphi^2\left(\frac{u_s}{u}\right)^2\,ds\\
&\quad+\int_0^a \varphi^2\left(k\mathrm{Ric}_h(\nu,\nu)+\sum_{j=2}^n R^h_{1j1j}\right)\,ds\\
&\quad+\int_0^a \varphi^2 \left(k|A|^2-A^2_{11}-\sum_{j=2}^n A^2_{1j}\right)\,ds,
\end{align*}
for every smooth function $\varphi$ such that $\varphi(0)=\varphi(a)=0$ and for every $k>0$.  Arguing as in \cite[Section 2]{ShYe} we have the following identity
\begin{align*}
k\mathrm{Ric}_h(\nu,\nu)+\sum_{j=2}^n R^h_{1j1j} &= k\mathrm{BRic}_h(e_1,\nu)-k\mathrm{Ric}_h(e_1,e_1)+kR^h_{1\nu1\nu}+\sum_{j=2}^n R^h_{1j1j}\\
&=k\mathrm{BRic}_h(e_1,\nu)+(1-k)\sum_{j=2}^n R^h_{1j1j}.
\end{align*}
Since $(N^{n+1},h)$ is with non-negative sectional curvature and either uniformly positive bi-Ricci curvature or uniformly positive Ricci curvature, we have $R^h_{1j1j}\geq 0$ for every $j=2,\ldots,n$ and either
$$
\mathrm{BRic}_h(e_1,\nu)\geq R_0\quad\text{ or }\quad \mathrm{Ric}_h(\nu,\nu)\geq R_0
$$
for some $R_0>0$. Therefore, if $k\leq 1$, we get
\begin{align*}
(n-1)\int_0^a (\varphi_s)^2\,ds &\geq k(n-3)\int_0^a \varphi\varphi_s\frac{u_s}{u}\,ds +\frac{k\left[4-k(n-1)\right]}{4}\int_0^a\varphi^2\left(\frac{u_s}{u}\right)^2\,ds\\
&\quad+\int_0^a \varphi^2 \left(kR_0+k|A|^2-A^2_{11}-\sum_{j=2}^n A^2_{1j}\right)\,ds.
\end{align*}
Since $A$ is trace-free, we have
$$
|A|^2\geq A_{11}^2+A_{22}^2+\ldots+A_{nn}^2+2\sum_{j=2}^n A_{1j}^2 \geq \frac{n}{n-1}A_{11}^2+2\sum_{j=2}^n A_{1j}^2,
$$
thus
$$
k|A|^2-A^2_{11}-\sum_{j=2}^n A^2_{1j}\geq \left(\frac{kn}{n-1}-1\right)A_{11}^2+(2k-1)\sum_{j=2}^n A^2_{1j}.
$$
Choose
$$
k=\frac{n-1}{n}\leq 1.
$$
We get
\begin{align*}
\int_0^a (\varphi_s)^2\,ds &\geq \frac{n-3}{n}\int_0^a \varphi\varphi_s\frac{u_s}{u}\,ds +\frac{6n-n^2-1}{4n^2}\int_0^a\varphi^2\left(\frac{u_s}{u}\right)^2\,ds+\frac{R_0}{n}\int_0^a \varphi^2\,ds.
\end{align*}
If $n\leq 5$, we have
$$
\frac{6n-n^2-1}{4n^2}\geq \delta_0>0.
$$
Moreover, there exists $C>0$, such that
$$
\frac{n-3}{n}\varphi\varphi_s\frac{u_s}{u}\geq -\delta_0 \varphi^2\left(\frac{u_s}{u}\right)^2-C(\varphi_s)^2.
$$
Therefore, there exists $C>0$, such that
$$
C\int_0^a (\varphi_s)^2\,ds \geq \frac{R_0}{n}\int_0^a \varphi^2\,ds
$$
for every smooth function $\varphi$ such that $\varphi(0)=\varphi(a)=0$. Integrating by parts we obtain
$$
\int_0^a \left(\varphi\varphi_ss+CR_0\varphi^2\right)ds\leq 0.
$$
Choosing $\varphi(s)=\sin(\pi s\,a^{-1})$, $s\in[0,a]$ one has
$$
\left(CR_0-\frac{\pi^2}{a^2}\right)\int_0^a\sin^2(\pi s\,a^{-1})ds\leq 0
$$
i.e.
$$
a^2\leq \frac{\pi^2}{CR_0}.
$$
We conclude that the length (in the metric $g$) of the geodesic $\widetilde{\gamma}(s)$ is finite and this gives a contradiction. Therefore $(M^n,g)$ must be compact and this concludes the proof of Theorem \ref{t2}.

\end{proof}

\begin{proof}[Proof of Corollary \ref{cor}]
If $M$ is stable, by Theorem \ref{t2} it must be compact.  Moreover there exists $u>0$ satisfying
$$
-\Delta u = \left[|A|^2+\ricc_h(\nu,\nu)\right] u\qquad\text{on } M.
$$
Integrating over $M$ we get a contradiction, since $\ricc_h>0$ on $M$. Equivalently, one can use $f\equiv 1$ in the stability inequality \eqref{eq-st} to get a contradiction.
\end{proof}

\

\

\

\begin{ackn}
\noindent The first and the second authors are members of the {\em GNSAGA, Gruppo Nazionale per le Strutture Algebriche, Geometriche e le loro Applicazioni} of INdAM. The third author is member of {\em GNAMPA, Gruppo Nazionale per l'Analisi Matematica, la Probabilit\`a e le loro Applicazioni} of INdAM. The authors thank Douglas Stryker for useful comments regarding Theorem \ref{t2}.
\end{ackn}

\

\

\noindent{\bf Data availability statement}

\noindent Data sharing not applicable to this article as no datasets were generated or analysed during the current study.

\

\

\

\

\

\end{document}